\documentclass[12pt]{amsart}

\usepackage[OT1]{fontenc}
\usepackage{type1cm}
\usepackage{enumitem}
\usepackage[UKenglish]{babel}
\usepackage{amssymb}
\usepackage{mathscinet}

\usepackage{tikz}

\usepackage{geometry}        
\usepackage{version}        

\usepackage{xcolor}
\usepackage[pagebackref]{hyperref}
\hypersetup{
   colorlinks,
    linkcolor={red!60!black},
    citecolor={blue!60!black},
    urlcolor={blue!90!black}
}

\numberwithin{equation}{section}

\theoremstyle{plain}
\newtheorem{theorem}{Theorem}[section]
\newtheorem*{theorem*}{Theorem}
\newtheorem{proposition}[theorem]{Proposition}
\newtheorem{lemma}[theorem]{Lemma}

\theoremstyle{remark}
\newtheorem{remark}[theorem]{Remark}

\newcommand{\PP}{\mathbb{P}}
\newcommand{\EE}{\mathbb{E}}
\newcommand{\R}{\mathbb{R}}
\newcommand{\N}{\mathbb{N}}

\DeclareMathOperator{\dist}{dist}
\DeclareMathOperator{\diam}{diam}
\DeclareMathOperator{\hdim}{dim_H}
\DeclareMathOperator{\CUT}{CUT}
\DeclareMathOperator{\LCUT}{LCUT}

\begin{document}

\title
{The largest slice of fractal percolation}

\author[P.~Shmerkin]{Pablo Shmerkin}
	\address[P.S.]{Department of Mathematics\\
		The University of British Columbia\\
		Room 121, 1984 Mathematics Road\\
		Vancouver, BC\\
		Canada V6T 1Z2
	}
	\email{pshmerkin@math.ubc.ca}
\author[V.~Suomala]{Ville Suomala}
\address[V.S.]{Research Unit of Mathematical Sciences,
        P.O.Box 8000, FI-90014,  University of Oulu,
    Finland}
    \email{ville.suomala@oulu.fi}

\thanks{P.S. was partially supported by an NSERC Discovery Grant. V.S. was partially supported by the Magnus Ehrnrooth Foundation} 
	

\begin{abstract}
For each $k\ge 3$, we determine the dimensional threshold for planar fractal percolation to contain $k$ collinear points. In the critical case of dimension $1$, the largest linear slice of fractal percolation is a Cantor set of zero Hausdorff dimension. We investigate its size in terms of generalized Hausdorff measures.
\end{abstract}

\maketitle

\section{Introduction}

We consider dyadic fractal percolation on the unit square $[0,1]^2$ defined as follows: Fix a parameter $p\in (0,1)$. We subdivide the unit square in $\R^2$ into $4$  sub-squares of equal size. We retain each such square with probability $p$ and discard it with probability $1-p$, with all the choices independent. For each of the retained squares, we continue inductively in the same fashion,  further subdividing them into $4$ sub-squares of equal size, retaining them with probability $p$ and discarding them otherwise, with all the choices independent of each other and the previous steps. The fractal percolation limit set $A=A^{\text{perc}(p)}$ is the set of points which are kept at each stage of the construction. 

In what follows, we will denote by $\mathcal{D}_n$ the family of closed dyadic sub-squares of the unit square of level $n$, that is
\[
    \mathcal{D}_n:=\left\{2^{-n}(j,i)+[0,2^{-n}]\times[0,2^{-n}]\,:\,0\le i,j\le 2^{n}-1\right\}\,.
\]
Let us also denote by $A_m$ the union of the squares in $\mathcal{D}_m$ that survive the first $m$ steps of the percolation, so that 
\[
    A=\bigcap_{m\in\N}A_m\,.
\]
Figure \ref{fig:perc} shows a realization of $A_7$ corresponding to $p=0.8$.

\begin{figure}
\label{fig:perc}
\begin{center}
    \includegraphics[width=0.7\textwidth]{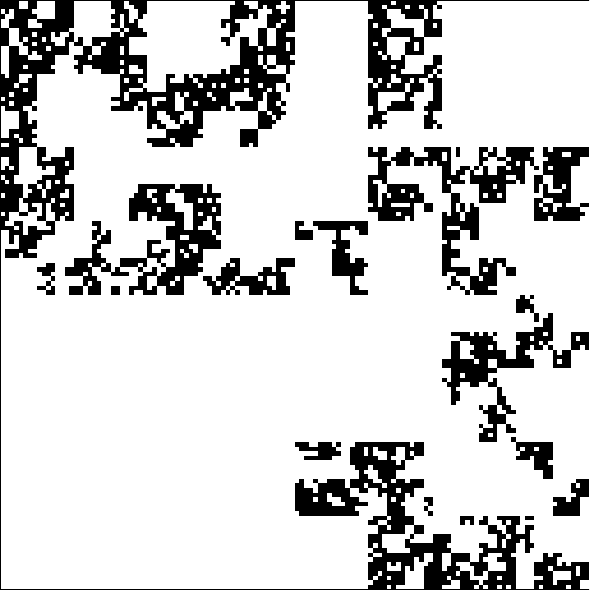}
\end{center}
\caption{A realization of fractal percolation with $p=0.8$.}
\end{figure}

Fractal percolation is one of the most popular models of random fractal sets, and its fractal geometric properties have been studied extensively in the literature, see e.g. \cite{FalconerGrimmett1992, AJJRS2012, BCJ2013, RamsSimon2014,  RamsSimon2014b, Don2015, RamsSimon2015, ShmerkinSuomala2015, PeresRams2016, CORS2017, ShmerkinSuomala2018, ShmerkinSuomala2020, RossiSuomala2020, Troscheit2020, KlattWinter2024}. It is well known that if $p\le \tfrac14$, then $A$ is almost surely empty, and if $\tfrac14<p<1$, there is a positive probability that $A\neq\varnothing$. Moreover, conditional on non-extinction (i.e $A\neq\varnothing$), $A$ is almost surely a fractal set of Hausdorff dimension
\begin{equation}\label{eq:s(d,p)}
\hdim A=2+\log_2 p\,,
\end{equation}
and the $2+\log_2 p$-dimensional Hausdorff measure of $A$ is almost surely zero.
See e.g. \cite[\S 14.3]{LyonsPeres2016} 
(which discusses a more general model).  We consider the following problem: 
\[
    \text{How large is the maximal intersection of }A\text{ with a line?}
\] 
If $p>1/2$, then $\hdim A>1$ and, conditional on non-extinction, there are many lines intersecting $A$ in a set of Hausdorff dimension $s=\hdim A-1$. This may be seen e.g. by considering intersections of $A$ with horisontal or vertical lines passing through $[0,1]^2$ which are fractal percolation on the line.
In terms of dimension, $\hdim A-1=1+\log_2 p$ is also the size of the largest slice since for this class of random sets, the Marstrand's slicing theorem takes a particularly strong form: There is a random variable $M$ such that
\begin{equation}\label{eq:unifrom_box_d}
\sup_{n\in\N,\,\ell\in\text{ lines}}\left|\left\{Q\in\mathcal{D}_n\,:\,\ell\cap Q\cap A_n\neq\varnothing\right\}\right|\le M 2^{sn}\,.
\end{equation}
Although the model there is slightly different, \eqref{eq:unifrom_box_d} follows from (the proof of) \cite[Theorem 3.1]{ShmerkinSuomala2015}. See also  \cite{RamsSimon2014, PeresRams2016, ShmerkinSuomala2018} for many related results. 

In this paper, we investigate the case $p\le 1/2$.  In this case $A$ has length zero (and so the same is true for all orthogonal projections), and therefore almost all lines do not intersect $A$ at all. This is why we look at the largest possible intersection with a line (as opposed to a typical intersection).

For each integer $k\ge 3$, there exists a critical threshold $p_k\in [\tfrac14,\tfrac12]$ such that if $p<p_k$, then almost surely, $A$ does not contain $k$ collinear points, and for $p>p_k$, almost surely conditioned on non-extinction, there is a line that contains $k$ collinear points; see Lemma \ref{lem:phase-transition} below. We are able to find the exact value of $p_k$ and show that at criticality (that is, $p=p_k$) $A$ does not contain $k$ collinear points.
\begin{theorem}  \label{thm:finite-intersections}
    Consider dyadic fractal percolation $A=A^{\text{perc}(p)}$ on $[0,1]^2$. Let $3\le k\in\N$.
\begin{enumerate}[label=\textbf{(\Alph*)},font=\upshape]
		\item \label{A} If $p\le 2^{(-k-2)/k}$, then almost surely, $A$ does not contain $k$ points on a line. 
		\item \label{B} If $p>2^{(-k-2)/k}$, then almost surely on $A\neq\varnothing$, there is a line $\ell$ such that  $|\ell\cap A|\ge k$.
		\end{enumerate}
\end{theorem}
 
Thus, $p_k$ is a strictly increasing sequence on $]\tfrac14,\tfrac12[$ and claims \ref{A} and \ref{B} thus give rise to infinitely many new phase transitions for fractal percolation: For $p\in]p_{k-1},p_k]$, there are lines that contain exactly $j$ points of $A$ for $j=0,\ldots,k-1$ but not for $j\ge k$. In particular, at criticality, that is for $p=p_k$, there are no lines intersecting $k$ points of $A$, see Remark \ref{rem:all_k}.

\begin{remark}
A remark in our previous paper \cite[Remark 6.15]{ShmerkinSuomala2020} claims the correct bound for $p_3$ (in all dimensions), but the proof suggested there is not correct as the set $V$ is claimed to be a plane, but is actually a non-linear surface. Although the paper \cite{ShmerkinSuomala2020} also develops methods for dealing with such non-linear intersections, rather than following this strategy we provide direct and more elementary proofs for the Claims \ref{A} and \ref{B} in Theorem \ref{thm:finite-intersections}.
\end{remark}

Finally, we turn to the case $p=\tfrac12=\lim_{k\to\infty} p_k$. It is relatively easy to see, using Theorem \ref{thm:finite-intersections}\ref{B}, that almost surely on $A\neq\varnothing$, there are lines that contain infinitely many points of $A$. The following result provides a more quantitative statement. Given  a gauge function $\phi$, let $\mathcal{H}^\phi$ denote the associated generalized Hausdorff measure  (see \S\ref{subsec:ghm} for the definition).

\begin{theorem} \label{thm:infinite-intersections}
    Consider dyadic fractal percolation $A=A^{\text{perc}(1/2)}$ on $[0,1]^2$.
\begin{enumerate}[label=\textbf{(\Alph*)},start=3,font=\upshape]
		\item \label{C} Almost surely, 
  \[\sup_{n\in\N,\,\ell\in\text{ lines}}\frac{\left|\left\{Q\in\mathcal{D}_n\,:\,\ell\cap Q\cap A_n\neq\varnothing\right\}\right|}{n^3}<\infty\,.\]
		\item \label{D} Almost surely on $A\neq\emptyset$, there is a  horizontal line $\ell$ with $\mathcal{H}^{\phi}(A\cap\ell)>0$, where 
  \[\phi(r)=\frac{\log\log(1/r)}{\log(1/r)}\,.\]
		\end{enumerate}
\end{theorem}
Note that Claim \ref{C} implies, in particular, that $\mathcal{H}^\phi(A)<\infty$ for $\phi(r)=\log^{-3}(1/r)$. Thus, the theorem provides the exact size of the largest linear slice of fractal percolation in the critical case $\hdim(A)=1$, up to powers of $\log(1/r)$. 

\begin{remark}
For any $m\in\N$, $m\ge 2$, one can define an $m$-adic version of fractal percolation by considering percolation on the $m$-adic sub-squares of $[0,1]^2$ instead of the dyadic ones, see \cite{RamsSimon2014}. We have assumed $m=2$ for simplicity of notation. All the above results generalize, with only trivial changes in the proofs, to an arbitrary base $m\ge 2$. In particular, for the $m$-adic planar fractal percolation, the threshold for $k$-points on a line is $p_k=m^{(-k-2)/k}$ which corresponds to the value $2-\frac{k+2}{k}$ for the Hausdorff dimension of the limit set $A$. Note that the latter value is independent of $m$.
\end{remark}

\subsection*{Acknowledgement}

We thank Andr\'as M\'ath\'e for useful discussions.

\section{Preliminaries}

\subsection{Notation}

We will use the standard big-$O$ notation: $f(n)=O(g(n))$ means that there is a constant independent of $n$ such that $f(n)\le C g(n)$ for all $n\in \N$. In particular, $f(n)=O(1)$ states that $f(n)$ is bounded. If we want to stress the dependence of the $O$-constant on some other parameters, we may use sub-indices, e.g. $f(n)=O_k (g(n))$ states that $f(n)\le C_k g(n)$ for all $n$, where $0<C_k<\infty$ is independent of $n$ but is allowed to depend on $k$. The notation $f(n)=\Omega(g(n))$ is used to denote $g(n)=O(f(n))$. Finally, $f(n)=\Theta(g(n))$, if $f(n)=O(g(n))=O(f(n))$. We let $|S|$ denote the cardinality of a set $S$.

In our arguments, we will consider the fractal percolation for a fixed parameter $\tfrac14<p\le\tfrac12$ and the induced probability $\PP$, a priori defined on $\{0,1\}^{\cup_{n\in\N}\mathcal{D}_n}$, but it may also be thought as a probability measure on the space of compact subsets of $[0,1]^2$. We denote by $\mathcal{B}_n$ the sigma-algebra generated by all possible choices of $A_k$, $0\le k\le n$, and given an event $\mathcal{E}$, let $\PP_\mathcal{E}$ and $\EE_\mathcal{E}$ denote the probability and expectation conditional on $\mathcal{E}$. If $Q,Q'\subset[0,1]$ are two dyadic squares, we denote by $I(Q,Q')$ the largest integer $n\ge 0$ such that $Q$ and $Q'$ are subsets of the same level $n$ dyadic square. 

\subsection{Generalised Hausdorff measures}
\label{subsec:ghm}

Let $\phi\colon[0,\infty[\to[0,\infty[$ be a non-decreasing function with $\phi(0)=0$. Such functions are known as \emph{gauge functions} or \emph{dimension functions}. The generalized Hausdorff measure $\mathcal{H}^\phi$ is defined as 
\[
\mathcal{H}^\phi(A)=\lim_{\delta\to 0} \mathcal{H}_{\delta}^{\phi}(A) = \sup_{\delta>0} \mathcal{H}_{\delta}^{\phi}(A)\,,
\]
where
\[ 
	\mathcal{H}_{\delta}^{\phi}(A) = \inf\left\{\sum_{i=1}^\infty\phi(\diam E_i)\,:\,A\subset\bigcup_{i=1}^\infty E_i\,,\diam E_i\le\delta\right\}\,.
\]

\subsection{The FKG inequality and zero-one law for inherited properties}

We recall the following version of the classical FKG inequality for later use, see \cite[Section 5.8]{LyonsPeres2016}.  In our context, an event $\mathcal{E}$ is decreasing if $\textbf{1}_{\mathcal{E}}(A)=1$ implies $\textbf{1}_{\mathcal{E}}(A')=1$ whenever $A,A'\subset[0,1]^2$ are compact and $A'\subset A$. 
\begin{lemma}\label{lem:FKG}
	If $\mathcal{E}_1,\ldots\mathcal{E}_N$ are decreasing events, then $\PP\left(\bigcap_{j=1}^N\mathcal{E}_j\right)\ge\prod_{j=1}^N\PP(\mathcal{E}_j)$.	
\end{lemma}

We call a property $P(X)$ of compact subsets $X$ of $[0,1]^2$ \emph{inherited} if the empty set satisfies $P$ and if $X$ satisfies $P$ then $X^Q$ satisfies $P$ for all $Q\in\mathcal{D}_n$ and $n\in\N$, where $X^Q= T_Q(X\cap Q)$, with $T_Q$ denoting the homothety that maps $Q$ to $[0,1]^2$. We have the following simple zero-one law for inherited properties; see \cite[Proposition 5.6]{LyonsPeres2016} (which is proved for Galton-Watson trees; the claim as stated below follows since fractal percolation is a geometric projection of a Galton-Watson tree as explained in \cite[\S 5.7]{LyonsPeres2016}).
\begin{lemma} \label{lem:inherited}
 Let $P$ be an inherited property and let $A$ be the fractal percolation limit set. Conditioned on $A\neq\emptyset$, $\PP(P(A))\in\{0,1\}$. 
\end{lemma}

As an immediate consequence we have:
\begin{lemma} \label{lem:phase-transition}
Fix $k\in\N$, $k \ge 3$. For each choice of $p$, either almost surely $A$ does not contain $k$ collinear points, or almost surely on $A\neq\emptyset$, there is a line containing $\ge k$ points of $A$.

Moreover, there exists a critical threshold $p_k\in [\tfrac14,\tfrac12]$ such that if $p<p_k$, then almost surely, $A$ does not contain $k$ collinear points, and for $p>p_k$, almost surely conditioned on non-extinction, there is a line that contains $k$ collinear points.
\end{lemma}
\begin{proof}
The property of \emph{not} containing $k$ collinear points is inherited, and so the first claim is immediate from Lemma \ref{lem:inherited}. The second claim follows from applying the first to a countable dense subset of $p\in [\tfrac14, \tfrac12]$.
\end{proof}

\section{Proof of Theorem \ref{thm:finite-intersections}}

\subsection{Discretization}\label{sec:disc}

We first reduce the problem of estimating $p_k$ to a discretized version which is easier to analyse. To that end, let us fix $m\in\N$ and $k$ disjoint squares $R_1,\ldots, R_k\in\mathcal{D}_m$ such that some line passes through the interior of all these squares. We emphasize that the $R_i$ are disjoint, not just non-overlapping, and hence since they are closed they are at distance $\ge 2^{-m}$ from each other.
 
 For each $n\in\N$, we will be considering level $m+n$ squares $Q_1,\ldots,Q_k\in\mathcal{D}_{n+m}$ such that $Q_i\subset R_i$ for each $i=1,\ldots,k$. Let 
 \[\mathcal{C}_n=\left\{\bigcup_{i=1}^k Q_i\,:\,Q_i\in\mathcal{D}_{n+m},Q_i\subset R_i,\exists\text{ a line }\ell\,,\ell\cap Q_i\neq\varnothing\,\forall i=1,\ldots,k\right\}\,.\]
Thus, each $C\in\mathcal{C}_n$ is a union of $k$ squares of level $m+n$, one from each of the initial squares $R_i$, and there is a line that intersects all the squares $Q_i$.
In what follows, when we write $C=Q_1\cup\ldots\cup Q_k\in\mathcal{C}_n$, we implicitly assume that $Q_i\in\mathcal{D}_{n+m}$ and that $Q_i\subset R_i$.

\begin{figure}[h]
\label{fig.C_n}
\begin{center}
\begin{tikzpicture}
\draw [color = gray] (0,0) rectangle (8,8);
\draw [thick] (0,0) rectangle (2,2);
\draw [thick] (6,6) rectangle (8,8);
\draw [thick] (2,4) rectangle (4,6);
\filldraw (0.5,1) rectangle (1,1.5);
\filldraw (7,7.5) rectangle (7.5,8);
\filldraw (3.5,4) rectangle (4,4.5);
\filldraw [fill = lightgray] (0,1.5) rectangle (0.5, 2);
\filldraw [fill = lightgray] (3,4) rectangle (3.5, 4.5);
\filldraw [fill = lightgray] (6.5,6.5) rectangle (7, 7);
\node[] at (1,0.7) {$R_1$};
\node[] at (3,5) {$R_2$};
\node[] at (7.4,6.8) {$R_3$};
\draw [dashed] (0,0.7) -- (7.3,8);
\draw [dashed] (0,1.7) -- (8,7.8);
\end{tikzpicture}
\end{center}
\caption{A possible selection of $R_i$ for $k=3$ and $m=2$. The unions of the squares of the same colour form two elements of $\mathcal{C}_2$.}
\end{figure}
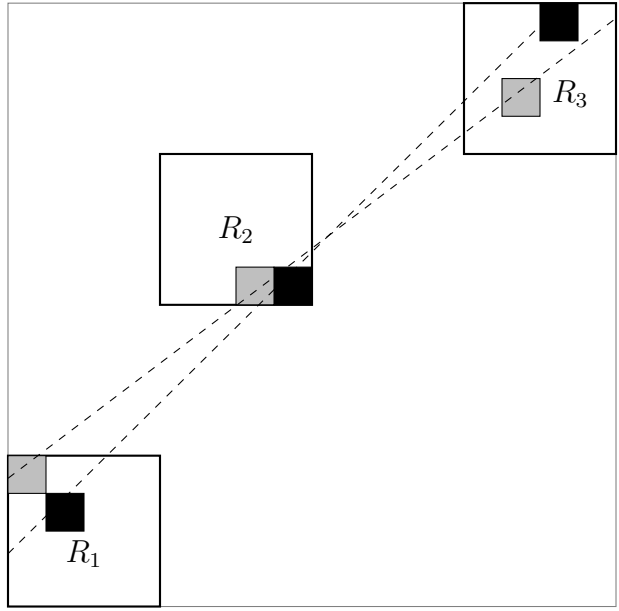

We will investigate the sequence of random variables
\begin{equation}
Y_n=\sum_{C\in\mathcal{C}_n}\textbf{1}\left[C\subset A_{m+n}\right]\,.
\end{equation}

Theorem \ref{thm:finite-intersections} will follow from the following discretized version of the statement.

\begin{proposition} \label{prop:discretized_statement}
	\begin{equation}\label{eq:discretized_statement}
		\PP(Y_n=0\text{ for some }n)=\begin{cases}1&\text{ if }p\le 2^{(-k-2)/2}\,,\\
		<1&\text{ if }p>2^{(-k-2)/2}\,.
		\end{cases}
		\end{equation}			
\end{proposition}

We now show how this implies Theorem \ref{thm:finite-intersections}, and defer the proof of the Proposition to \S \ref{subsec:A} and \S \ref{subsec:B}.

\begin{proof}[Proof of Theorem \ref{thm:finite-intersections} assuming Proposition \ref{prop:discretized_statement}]

We first show Claim \ref{A}. Suppose $p\le 2^{(-k-2)/k}$. There are only countably many possible choices for the initial rectangles $R_1,\ldots,R_k$, and for any $k$ collinear points $x_1,\ldots, x_n\in A$ there are such $R_1,\ldots, R_k$ for which $x_i\in R_i $, implying that $Y_n\ge 1$ for all $n$.

Now, we prove Claim \ref{B}. If $Y_n\neq 0$ for all $n$, then for each $n$, there is a line $\ell_n$ and points $x_{n,1},\ldots,x_{n,k}\in\ell\cap A_{m+n}$ with mutual distances at least $2^{-n}$. Using  compactness,  this implies that some line $\ell$ hits (at least) $k$ points of $A$. By Lemma \ref{lem:phase-transition}, positive probability implies full probability on non-extinction.
\end{proof}

We stress that from now on, $m\in\N$, $3\le k\in\N$, and the squares $R_1,\ldots,R_k$ are fixed in such a way that some line $\ell$ passes through the interior of all the squares $R_i$. We condition on the event that all $R_i$ are in $A_m$; note this does not affect the claim of Proposition \ref{prop:discretized_statement}. More precisely, we denote $\PP=\PP_{R_1,\ldots,R_k\subset A_m}$ and $\EE=\EE_{R_1,\ldots, R_k\subset A_m}$. Note that the sets $A_{m+n}\cap R_j$, $j=1,\ldots, k$, are now mutually independent for all $n\ge 1$. 

The following counting lemma will be useful later.
\begin{lemma}\label{lem:card_C_n}
$|\mathcal{C}_n|=\Theta(2^{(k+2)n})$.
\end{lemma}

\begin{proof}
Consider $Q_1,Q_2\in\mathcal{C}_n$, $Q_1\subset R_1$, $Q_2\subset R_2$. For each $i=3,\ldots,k$, the family
\[
\bigl\{Q\in\mathcal{D}_{m+n}\,,|\,Q\subset R_i\,,\exists\text{ a line }\ell\text{ such that }\ell\cap Q_1\neq\varnothing,\ell\cap Q_2\neq\varnothing,\ell\cap Q\neq\varnothing\bigr\}\,,
\]
consists of at most $O(2^n)$ squares. Indeed, each square from this family is a subset of an $O(2^{-n})$ tubular neighbourhood of a fixed line. It follows that having selected $Q_1$ and $Q_2$, the family
\[
    \bigl|\bigl\{(Q_1,Q_2,Q_3,\ldots,Q_k)\in\mathcal{C}_n\bigr\}\bigr|
\]
has at most $O(2^{(k-2)n})$ elements. Since there are $2^{4n}$ possible ways to pick the squares $Q_1$ and $Q_2$, we conclude $|\mathcal{C}_n|=O(2^{(k+2)n})$.

To estimate $|\mathcal{C}_n|$ from below, pick a line $\ell$ that passes through the interior of the squares $R_i$. Then, there is an $\varepsilon>0$ such that also each line $\ell'$ which is $(2^{m+2}\varepsilon)$-close to $\ell$ in the Hausdorff distance of $[0,1]^2$ intersects each $R_i$ and, moreover, the length of any of the line segments $\ell'\cap R_i$ is $\Omega_{R_1,\ldots,R_k}(1)$. 

Let $E=\ell(\varepsilon)$ denote a tubular neighbourhood of $\ell$ of width $2\varepsilon$. 
Suppose $2^{1-n}<\varepsilon$. Let us pick $Q_1,Q_2\in\mathcal{C}_n$ such that $Q_1\subset E\cap R_1$, $Q_2\subset E\cap R_2$ and let $\ell'$ be a line that passes through $Q_1$ and $Q_2$. Since the distance between $R_1$ and $R_2$ is at least $2^{-m}$, it follows that $\ell$ and $\ell'$ are $(2^{m+2}\varepsilon)$-close and thus the length of each of the line segments $R_3\cap\ell',\ldots,R_k\cap\ell'$ is $\Omega(1)$. Consequently, there are $\Omega(2^{(k-2)n})$ different ways to pick squares $Q_3,\ldots, Q_k\in\mathcal{D}_{m+n}$, $Q_i\subset R_i$, along the line $\ell'$. Finally, since there are $\Omega(2^{4n})$ admissible choices for $Q_1$ and $Q_2$, it follows that $|\mathcal{C}_n|=\Omega(2^{(k+2)n})$. Note that we have assumed that $2^{1-n}<\varepsilon=\varepsilon(R_1,\ldots,R_k)$, but by enlarging the $O$-constants, the estimate remains valid for all $n$.
\end{proof}

\subsection{The case $p\le 2^{(-k-2)/k}$}
\label{subsec:A}

The proof of Proposition \ref{prop:discretized_statement} in this case (which implies Claim \ref{A}) is an application of the first moment method. We borrow some ideas from \cite[Proposition 5.13]{ShmerkinSuomala2020} to cover also the critical case $p=2^{(-k-2)/k}$. We note that in the subcritical case $p<2^{(-k-2)/k}$, it is possible to use the first moment method more directly and avoid the use of the FKG inequality altogether.

Given $C=Q_1\cup\ldots\cup Q_k\in\mathcal{C}_n$, where $Q_i\in\mathcal{D}_{n+m}$, we have $\PP(Q_i\subset A_{m+n})=p^n$ for all $i=1,\ldots,k$. Since the events $Q_i\subset A_{m+n}$ are independent, we infer that
$\PP(C\subset A_{m+n})=p^{kn}$.
Combining this with Lemma \ref{lem:card_C_n}, we have
\begin{equation}\label{eq:ExpY_n}
\EE(Y_n)=\Theta\left(2^{n(k+2)}p^{kn}\right)\,.
\end{equation}
(This holds regardless of the value of $p$; later we will use it also in the super-critical case.) By Markov's inequality, we thus have for each $M\ge 1$,
\begin{equation}
\begin{split}\label{eq:Markov}
\PP(Y_n\ge M)\le\frac{\EE(Y_n)}{M}
=\frac{O\left(2^{(k+2)n}p^{kn}\right)}{M}=\frac{O(1)}{M}\,,
\end{split}
\end{equation} 
using the assumption $p\le 2^{(-k-2)/k}$.

For each $C=Q_1\cup\ldots\cup Q_k\in\mathcal{C}_{n}$, we have the simple estimate 
\begin{equation}\label{eq:trivial_lb}
\begin{split}
&\PP\left(\sum_{\widetilde{C}\in\mathcal{C}_{n+1}\,,\widetilde{C}\subset C}\textbf{1}[\widetilde{C}\subset A_{m+n+1}]=0\,|\,C\subset A_{n+m}\right)\\
&\ge\PP\left(A_{m+n+1}\cap Q_i=\varnothing \text{ for some }i=1,\ldots, k\,|\,C\subset A_n\right)\\
&=1-(1-(1-p)^4)^k=:q>0\,.
\end{split}
\end{equation}

Conditional on $\mathcal{B}_{n+m}$, let $\{C_j=Q^j_1\cup\ldots\cup Q^{j}_k\in\mathcal{C}_n\}_{j=1}^{Y_n}$ constitute the configurations $C_j\in\mathcal{C}_n$ with $C_j\subset A_{n+m}$, and denote by $\mathcal{E}_j$ the event
\[Q^j_i\cap A_{m+n+1}=\varnothing\text{ for some }i=1,\ldots,k\,.\]
Since these are all decreasing events, the FKG inequality (Lemma \ref{lem:FKG}) combined with \eqref{eq:trivial_lb} yields
\begin{equation}\label{O4}
\PP(Y_{n+1}=0\,|\,\mathcal{B}_{n+m})=\PP\left(\bigcap_{j=1}^{Y_n}\mathcal{E}_j\,|\,\mathcal{B}_{n+m}\right)\ge q^{Y_n}\,.
\end{equation}

We wish to show that $\lim_n Y_n=0$ almost surely. Since the events $Y_n>0$ are nested  (i.e. $Y_n=0$ implies $Y_{n+1}=0$), it is enough to verify $\PP\left(Y_n>0\right)\rightarrow 0$ as $n\to\infty$. Assume on the contrary that $\PP(Y_n>0)>\eta>0$ for all $n$. Writing
\[
\PP(Y_n>0) = \PP(Y_n>0\,|\,{Y}_{n-1}>0)\times\cdots\times\PP(Y_2>0\,|\,Y_1>0)\times \PP(Y_1),
\]
we see that it is enough to show that $\PP(Y_{n+1}=0\,|\,Y_n>0)\ge c$ for some constant $c>0$ independent of $n$. By \eqref{eq:Markov}, we may pick $M$ so large that $\PP(0< Y_n\le  M) \ge \eta/2$ for all $n$. Combining this with \eqref{O4},
\[\PP(Y_{n+1}=0\,|\,Y_n>0)\ge\PP(Y_{n+1}=0\,|\,1\le Y_n\le M)\PP(Y_n\le M\,|\,Y_n>0)\ge \eta q^{M}/2\,,\]
as required.

\subsection{The case $p>2^{(-k-2)/k}$}
\label{subsec:B}

Throughout this section, we assume that 
\begin{equation}\label{eq:tech_up}
2^{k/(1-k)}> p>2^{(-k-2)/k}\,.
\end{equation} 
Imposing the upper bound is not a restriction, since 
\[
	\PP_p(A\text{ contains at least }k\text{-many points on a line})\]
is clearly non-decreasing in $p$. The goal is to show that, with positive probability, $Y_n>0$ for all $n$. Recall that this implies Claim \ref{B}. We use the second moment method; the required estimate on the variance is as follows:
\begin{theorem}\label{thm:reduced_main'}
	$\EE(Y_n^2)=O\left(2^{(2k+4)n}p^{2kn}\right)$.
\end{theorem}

Proposition \ref{prop:discretized_statement} follows from this by a standard argument. Indeed, we only have to show that, with positive probability,  $Y_n\not\rightarrow 0$. Assume on the contrary that $Y_n\rightarrow 0$ almost surely. Consider the normalized sequence $Z_n= 2^{-n(k+2)}p^{-kn}Y_n$. It follows from  \eqref{eq:ExpY_n} that 
\begin{equation} \label{eq:ExpZ_n}
\EE(Z_n)=\Theta(1)\,.
\end{equation}
Since $p>2^{(-k-2)/k}$, we also have $Z_n\rightarrow 0$ almost surely. Given $M>0$, write
\begin{align*}
\EE(Z_n)=\EE\left(Z_n\textbf{1}[Z_n\le M]\right)+\EE\left(Z_n\textbf{1}[Z_n>M]\right)\,.
\end{align*}
Here $\lim_n\EE\left(Z_n\textbf{1}[Z_n\le M]\right)=0$ by  bounded convergence, whereas by Theorem \ref{thm:reduced_main'}
\[\EE\left(Z_n\textbf{1}[Z_n>M]\right)\le\frac{\EE(Z_n^2)}{M}=O\left(\frac1M\right)\longrightarrow 0\,,\]
as $M\to\infty$. We conclude that $\EE(Z_n)\rightarrow 0$, contradicting \eqref{eq:ExpZ_n}.

Towards the proof of Theorem \ref{thm:reduced_main'}, our main task is to control the correlations between the events $C\subset A_{n+m}$ and $C'\subset A_{n+m}$ for disjoint pairs $C,C'\in\mathcal{C}_n$. To that end, let us fix $C'=Q'_1\cup\ldots\cup Q'_k\in\mathcal{C}_n$ and for $i=1,\ldots,k$ let
$d_i=d_i(C,C')=I(Q_i,Q'_i)-m$ be the maximal integer such that $Q_i$ and $Q'_i$ belong to the same dyadic cube of level $m+d_i$. Let 
\[
	n\ge d^1(C,C')\ge d^2(C,C')\ge\ldots\ge d^k(C,C')\ge 0
\] 
denote the $d_i$:s in decreasing order. Given $1\le l\le s\le n$ and $0\le r\le (k-2)l$, define $\Phi_{s,l,r}=\Phi_{s,l,r}(C')$ as
\[
	\Phi_{s,l,r}=\left\{C\in\mathcal{C}_n\,:\,d^2(C,C')=l,\,d^1(C,C')=s,\,\sum_{i=1}^k d_i(C,C')-l-s=r\right\}\,.
\]

We proceed by estimating the size of each $\Phi_{s,l,r}$ and the probability of $C\subset A_{m+n}$ for $C\in\Phi_{s,l,r}(C')$, conditional on $C'\subset A_{m+n}.$
\begin{lemma}\label{lemma:size_Phi}
	We have $|\Phi_{s,l,r}|=O_k(2^{(k+2)n-2s-2l-r})$ for all $C'=(Q'_1,\ldots,Q'_k)\in\mathcal{C}_n$.
\end{lemma}

\begin{proof}
	We refine the argument from the proof of Lemma \ref{lem:card_C_n}. After reordering the squares $R_i$, we may assume that for each $C\in\mathcal{C}_n$, $d_j(C,C')=d^j(C,C')$. That is, the $j$-th smallest dyadic distance between $Q_i$ and $Q'_i$ is realized inside $R_j$.
	Since there are $k!=O_k(1)$ ways to reorder the sequence $R_1,\ldots,R_k$, this assumption does not affect the generality.
	  
	Thus, for each $i=1,\ldots,k$, the squares $Q_i$ and $Q'_i$ belong to a same dyadic square of level $m+d_i$. Moreover, if $C\in\Phi_{s,l,r}$, we have $d_1=s$, $d_2=l$, and $\sum_{i=3}^k d_i=r$.
	
	Let us now pick $Q_1,Q_2\in\mathcal{D}_{m+n}$, $Q_1\subset R_1$, $Q_2\subset R_2$ such that $Q_1$ and $Q'_1$ are sub-squares of a same square in $\mathcal{D}_{m+s}$ and $Q_2$ and $Q'_2$ are sub-squares of a same square in $\mathcal{C}_{m+l}$.
	The family of $Q\in\mathcal{D}_{m+n}$ satisfying
    	\[\,I(Q,Q'_i)=m+d_i\,,\exists\text{ a line }\ell\text{ with }\ell\cap Q_1\neq\varnothing,\ell\cap Q_2\neq\varnothing,\ell\cap Q\neq\varnothing\,,\]
	consists of at most $O(2^{n-d_i})$ elements for each $i=3,\ldots,k$, simply because each of its elements must intersect a fixed tubular neighbourhood of a line segment of length $O(2^{-d_i})$ and of width $O(2^{-n})$. Combining this with the fact that there are $\le 2^{2(n-s)}$ possible choices for $Q_1$ and $\le 2^{2(n-l)}$ choices for $Q_2$, we arrive at
	\[
		|\Phi_{s,l,r}|=O\left(2^{2(n-s)+2(n-l)+\sum_{i=3}^k(n-d_i)}\right)=O\left(2^{(k+2)n-2s-2l-r}\right)\,,
	\]
	as required.
\end{proof}

\begin{lemma}\label{lem:effect_of_correlations}
	If $C\in\Phi_{s,l,r}(C')$, then
	$\PP(C\subset A_{n+m}\,|\,C'\subset A_{n+m})= p^{kn-s-l-r}$.
\end{lemma} 

\begin{proof}
Recall that if $C=(Q_1,\ldots,Q_k)\in\Phi_{s,l,r}(C')$ for $C'=(Q'_1,\ldots,Q'_k)$, then $\sum_{i=1}^k d_i=s+l+r$, where $d_i=I(Q_i,Q'_i)-m$. Thus, there are squares $D_i\in\mathcal{D}_{m+d_i}$ such that $Q_i,Q'_i\subset D_i$, but conditional on $D_i\subset A_{m+d_i}$, the events $Q_i\subset A_{m+n}$, $Q'_i\subset A_{m+n}$ are independent. Recalling once more that as we are conditioning on $R_1\cup\ldots\cup R_k\subset A_m$, the laws of $A_{m+n}\cap R_i$ are independent, it follows that
\[\PP(Q_i\subset A_{m+n}\,|\,C'\subset A_{m+n})=\PP(Q_i\subset A_{m+n}\,|\,D_i\subset A_{m+d_i})=p^{n-d_i}\,.\]
Using again the independence of $A_{m+n}\cap R_i$, $i=1,\ldots,k$,
\[\PP(C\subset A_{m+n}\,|\,C'\subset A_{m+n})=p^{kn-\sum_{i=1}^k d_i}=p^{kn-s-l-r}\,,\]
which is what we want.
\end{proof}

\begin{proof}[Proof of Theorem \ref{thm:reduced_main'}]
For each $C'\in\mathcal{C}_{n}$, we use Lemmas \ref{lemma:size_Phi} and \ref{lem:effect_of_correlations} to estimate
\begin{align*}
\EE_{C'\subset A_{m+n}}\left(\sum_{C\in\mathcal{C}_n}\textbf{1}[C\subset A_{n+m}]\right)&\le\sum_{s=0}^n\sum_{l=0}^s\sum_{r=0}^{(k-2)l}|\Phi_{s,l,r}(C')|p^{kn-s-l-r}\\
&O\left(2^{(k+2)n}p^{kn}\right)\sum_{s=0}^n2^{-2s}p^{-s}\sum_{l=0}^s2^{-2l}p^{-l}\sum_{r=0}^{(k-2)l}2^{-r}p^{-r}\\
&=O\left(2^{(k+2)n}p^{kn}\right)\sum_{s=0}^n2^{-2s}p^{-s}\sum_{l=0}^s2^{-kl}p^{(1-k)l}\\
&=O\left(2^{(k+2)n}p^{kn}\right)\sum_{s=0}^n2^{(-k-2)s}p^{-ks}\\
&=O\left(2^{(k+2)n}p^{kn}\right)\,,
\end{align*}
since we are assuming that $2^{(-k-2)/k}<p<2^{k/(1-k)}<\tfrac12$.

Using this and \eqref{eq:ExpY_n}, we conclude
\begin{align*}
\EE(Y_n^2)&=\sum_{C'\in\mathcal{C}_n}\PP(C'\subset A_{m+n})\EE_{C'\subset A_{m+n}}\left(\sum_{C\in\mathcal{C}_n}\textbf{1}[C\subset A_{n+m}]\right)\\
&=O\left(2^{(2+k)n}p^{kn}\right)\EE(Y_n)\\
&=O\left(2^{(4+2k)n}p^{2kn}\right)\,,
\end{align*}
and this completes the proof.
\end{proof}

\begin{remark}\label{rem:all_k}
If $p_m<p<\tfrac12$, then for all $1\le k\le m$, there are lines containing exactly $k$ points of $A$. This does not follow directly from Theorem \ref{thm:finite-intersections} as stated, but it can be deduced using a slight variant of the proofs. Indeed, we may choose $\varepsilon>0$ such that $p^{1-\varepsilon}<\tfrac12$ and remove from $\mathcal{C}_n$ those $k$-tuples $C=Q_1\cup\ldots\cup Q_k$ such that some line passing through $Q_1,\ldots,Q_k$ intersects a square $Q\in\mathcal{D}_{m+n}$ such that $Q\subset A_{n+m}$ and $\dist(Q,C)>2^{-\varepsilon n}$. The proof goes through with this modified collection $\mathcal{C}_n$. We leave the details to the interested reader.
\end{remark}

\section{Proof of Theorem \ref{thm:infinite-intersections}}

\subsection{Proof of Claim \ref{C}}

\label{subsec:C}

It suffices to show that 
\begin{equation}\label{eq:sup_length}
    \sup_{\ell}\lambda(A_n\cap\ell)=O(n^3 2^{-n})\text{ almost surely,}
\end{equation}
where the supremum is over all lines and $\lambda$ denotes Lebesgue measure on the line $\ell$.
Indeed, if $\ell$ intersects $\ge c$ squares $Q\in\mathcal{D}_n$, $Q\subset A_n$, there is a nearby line $\ell'$ for which 
$\lambda(A_n\cap\ell')=\Omega(c 2^{-n})$.

In order to prove \eqref{eq:sup_length}, let us momentarily fix $M\ge 1$ and let $\mathcal{A}_n$ denote the event 
\[
	\sup_{\ell}\lambda(A_n\cap\ell)<M\, n^3 \, 2^{-n}\,.
\] 
Conditional on $\mathcal{A}_n$, for each fixed line, \cite[Lemma 3.4]{ShmerkinSuomala2018} yields
\[
	\PP\left(\lambda(A_{n+1}\cap\ell)>\left(M(n+1)^3-1\right)2^{-n-1}\right)=O(\exp(-\Omega(1)M n))\,.
\]
On the other hand, it is not hard to see (\cite[Lemma 3.3]{ShmerkinSuomala2015}) that there is a finite deterministic family of lines $\mathcal{L}_n$, with $|\mathcal{L}_n|=\exp(O(n))$, such that
\[
	\sup_{\ell\in\text{lines}}\lambda(A_n\cap\ell)\le\sup_{\ell\in\mathcal{L}_n}\lambda(A_n\cap\ell)+2^{-n-1}\,.
\]
Combining the last two estimates yields $\PP(\mathcal{A}^c_{n+1}\,|\,\mathcal{A}_n)=O(\exp(-\Omega(Mn))$. By virtue of the Borel-Cantelli lemma, we then have the estimate
\[\PP\left(M<\sup_{\ell\in\text{lines},\, n\in\mathbb{N}}n^{-3}\,2^n\,\lambda(A_n\cap\ell)\right)=O(\exp(-\Omega(M))\,.\]
In particular, this implies \eqref{eq:sup_length}.

\subsection{Proof of Claim \ref{D}}

We start by introducing a class of Cantor constructions on $[0,1]$ with respect to relative scales $m_1,m_2,\ldots\in\N$. More precisely, at stage $n$ of the construction, we will work with dyadic intervals of size $2^{-s_n}$, $s_n=m_1+\ldots+m_n$ (with $s_0=0$). In order to achieve our goal, we choose these numbers so that $\frac{m_n}{\log m_n}=\Theta(2^n)$ or, equivalently, $2^n\phi(2^{-s_n})=\Theta(1)$ (where $\phi$ is as in the statement of Theorem \ref{thm:infinite-intersections}). For our computations, it will be convenient to make these assumptions slightly more quantitative:
we assume that $m_n$ is a non-decreasing sequence of integers satisfying
\begin{equation}\label{eq:m_n}
m_n\ge 2^{n+4}\,,\quad\quad\text{and}\quad\quad
c\le\frac{2^{n+1} \log_2(m_{n+1}+1)}{m_{n}}\le 1
\end{equation}
for all $n\ge -2$, where $c>0$ is a constant.
(We define the numbers $m_{-2}, m_{-1},m_0\in\N$ for technical reasons, these will not be used in the construction of the set.) A concrete example sequence satisfying these assumptions is given by $m_n=(n+3)2^{n+4}$.

In the first stage of the construction, we pick one dyadic interval of size $2^{-m_1}$ inside $[0,\tfrac12]$, and another one inside $[\tfrac12,1]$. At each step $n$, we have a union of $2^n$ dyadic intervals of size $2^{-s_n}$ and each of these intervals $I=[q,q+2^{-s_n}]$ is replaced by two dyadic subintervals of length  $2^{-s_{n+1}}$, such that one of the chosen intervals is contained in $[q,q+2^{-s_n-1}]$, and the other one is contained in $[q+2^{-s_n-1},q+2^{-s_n}]$. 
A discretized level $n$ Cantor set is the union of the thus obtained $2^n$ closed dyadic intervals of length $2^{-s_n}$. We let $\Gamma_n$ denote all such discretized level $n$ Cantor sets on $[0,1]$ and, moreover, let 
\[
	\mathcal{C}_n=\Gamma_n\times\left\{(k+\tfrac12) 2^{-s_n}\,:\,0\le k\le 2^{s_n}-1\right\}
\]
consist of copies of such discretized Cantor sets on the horizontal lines 
\[
	\mathcal{L}_n=\left\{\R\times \{(k+\tfrac12) 2^{-s_n}\}\,:\,0\le k\le 2^{s_n}-1\right\}\,.
\] 
We let $\ell_C\in\mathcal{L}_n$ denote the line containing $C$, for $C\in\mathcal{C}_n$.

Observe that
\begin{align*}
|\mathcal{C}_n|&=2^{s_n}\times2^{2(m_1-1)}\times 2^{4(m_2-1)}\times\ldots\times2^{2^{n}(m_n-1)}\\
&=2^{2-2^{n+1}+\sum_{j=1}^n (2^j+1) m_j}\,.    
\end{align*}
For each $C\in\mathcal{C}_n$, 
\[
	\PP(C\subset A_{s_n})=2^{-\sum_{j=1}^n 2^j m_j}
\]
and thus the random variable
$Y_n=\sum_{C\in\mathcal{C}_n}\textbf{1}[C\subset A_ {s_n}]$ satisfies
\begin{equation} \label{eq:expect-Cantor-Yn}
  \EE(Y_n) = 	2^{2-2^{n+1}+s_n}.
\end{equation}
It now follows from \eqref{eq:m_n} that $\EE(Y_n)\rightarrow\infty$ as $n\to\infty$. The key to the proof is the following correlation estimate.
\begin{proposition} \label{prop:correlation-Cantor}
For a fixed $C\in\mathcal{C}_n$,
\begin{equation}\label{eq:ultimate_goal}
	\sum_{C'\in\mathcal{C}_n}\PP\left(C'\subset A_{s_n}\,|\,C\subset A_{s_n}\right)=O(1)2^{-2^{n+1}+s_n}\,.   
	\end{equation}		
\end{proposition}
\begin{proof}
Fix $C\in\mathcal{C}_n$ and let $\ell=\ell_C\in\mathcal{L}_n$.

Fix $1\le k\le n$ and $1\le p\le m_{k}$ for the time being. We consider lines $\ell'\in\mathcal{L}_n$ whose dyadic distance to $\ell$ is $2^{p-s_k}$. Suppose that $C'\in\mathcal{C}_n$ is contained in $\ell'$. We next explain how to determine the correlation structure between the events $C\subset A_{s_n}$, $C'\subset A_{s_n}$.

We let $T(k)$ denote the level $k$ dyadic tree, with root $\varnothing$. For each vertex $v$, we label one of its children as ``left'' and the other one as ``right''. Moreover, we mark the vertices $v\in T(k)$ with the construction intervals of $C$ in a natural way: the $2^j$ vertices of $T(k)$ at distance $|v|=j$ from the root are marked with the construction intervals $I_{v}(C)$ of $C$ of length $2^{-s_j}$ from left to right with respect to the lexicographic order of the tree arising from the ``left/right''-labelling. 

Recall that a cutset of a tree is a set of vertices such that every path from the root to a leaf intersects the cutset at a unique vertex. We define a cutset $\CUT(C,C')$ of $T(k)$ as follows. Let $u^{-}$ denote the parent of $u\in T(k)\setminus\varnothing$. Then $u\in\CUT(C,C')$ if and only if the intervals $I_u(C)$ and $I_u(C')$ fall into different level $s_{|u|}$ dyadic squares (note this implies $u\neq\varnothing$), while $I_{u^{-}}(C)$ and $I_{u^{-}}(C')$ are contained in the same level $s_{|u|-1}$ dyadic square. Since, by assumption, the dyadic distance between $\ell$ and $\ell'$ is $>2^{-s_k}$, this is indeed a  cutset of $T(k)$.

Let $u\in\CUT(C,C')$. By definition, there exists $\lambda_u\in\{1,\ldots,m_{|u|}\}$ such that the intervals $I_u(C)$ and $I_u(C')$ split at the dyadic scale $s_{|u|}-\lambda_u$. (More precisely, the dyadic squares in $\mathcal{D}_{s_{|u|}-j}$ containing $I_u(C)$ and $I_u(C')$ are the same for $j=\lambda_u$ and different for $j=\lambda_u-1$.)  We refer to the pair $(\CUT(C,C');\lambda)$ as the labelled cutset $\LCUT(C,C')$.

In terms of this labelling, the probability that $C'$ survives given that $C$ survives is given by
\begin{align}\label{eq:cond_P}
    \PP(C'\subset A_{s_n}\,|\,C\subset A_{s_n})=\prod_{u\in \CUT(C,C')} 2^{-\lambda_u-\sum_{j=|u|+1}^n  2^{j-|u|}m_j}\,.
\end{align}
This follows since the process inside each of the squares corresponding to the cutset is independent of the process in all the other squares.

Let $(K,\lambda)$ be a labelled cutset of the $T(k)$ tree (with $1\le \lambda_{u}\le m_{|u|}$). We aim to count how many elements $C'\in\mathcal{C}_n$ satisfy $\LCUT(C,C')=(K,\lambda)$ and $d(\ell_C,\ell_{C'})=2^{p-s_k}$. On the one hand, there are 
\[
	2^{m_{k+1}+\ldots+m_n+p}
\]
choices for the line $\ell_{C'}$. (Note that we interpret this as $2^p$ if $k=n$.)
On the other hand, once the line $\ell_{C'}$ is fixed, given a vertex $u\in K$, there are 
\[
	2^{\lambda_u+2-2^{n-|u|+1}+\sum_{j=|u|+1}^n 2^{j-|u|}m_j}
\]
choices for $(Q_u, C'\cap Q_u)$, where $Q_u$ is the dyadic square containing $I_u(C')$. Thus, there are  
\begin{equation} \label{eq:sets-for-given-LCUT}
    2^{m_{k+1}+\ldots+m_n+p}\prod_{u\in K}2^{\lambda_u+2-2^{n-|u|+1}+\sum_{j=|u|+1}^n 2^{j-|u|}m_j}
\end{equation}
elements $C'\in\mathcal{C}_n$ with $d(\ell_C,\ell_{C'})=2^{p-s_k}$ and $\LCUT(C,C')=(K,\lambda)$. 

Note that $K$ can be formed by starting with $K_0=\{\varnothing\}$ and then replacing a vertex by its two children a finite number of times to obtain cutsets $K_0,K_1,\ldots,K_M=K$. Each such replacement step preserves the value of
\[
	\sum_{u\in K_m} 2^{n-|u|+1}\,.
\]
Therefore, since $|K|\le 2^{k}$, we have
\begin{equation} \label{eq:prod-over-K}
	\prod_{u\in K} 2^{2-2^{n-|u|+1}}= 2^{2|K|}2^{-2^{n+1}} \le 2^{2^{k+1}}\cdot 2^{-2^{n+1}}\,.
\end{equation}
Combining \eqref{eq:cond_P}, the count \eqref{eq:sets-for-given-LCUT} of Cantor sets corresponding to a given cutset, and  \eqref{eq:prod-over-K}, it follows that the expected number of $C'\subset A_{s_n}$ with $\LCUT(C,C')=(K,\lambda)$, $d(\ell_C,\ell_{C'})=2^{p-s_k}$, conditional on $C\subset A_{s_n}$, is 
\[
2^{m_{k+1}+\ldots+m_n+p}\prod_{u\in K}2^{2-2^{n-|u|+1}}
\le 2^{-2^{n+1}+2^{k+1}+m_{k+1}+\ldots+m_n+p}\,.
\]

We can crudely bound the number of cutsets of $T(k)$ by $2^{2^k}$ (for example, this is larger than the number of subsets of $T(k-1)$). Given a cutset $K$ of $T(k)$, the number of labels $\lambda$ satisfying $1\le\lambda_u\le m_{|u|}\le m_{k}$ is bounded (again, crudely) by $m_k^{2^{k}}$. Thus, the total number of labelled cutsets $(K,\lambda)$ of $T(k)$ is smaller than $2^{2^k}m_k^{2^{k}}$.

Finally, we treat the two lines $\ell'$ in $\mathcal{L}_n$ with $d(\ell',\ell_C)\le 2^{-s_n}$. Then $k=n$ and the only difference to the $d(\ell,\ell')>2^{-s_n}$ is that we also need to take into account the case $p=0$ and allow the possibility $\lambda_u=0$ if $|u|=n$. The upper-bound for the number of cut-sets remains the same while for the number of labellings of a fixed cut-set, the crude bound gets replaced by $(m_n+1)^{2^n}$. 

All in all, adding up over all possible values of $k$ and $p$, yields the upper bound
\begin{align*}
\sum_{C'\in\mathcal{C}_n}\PP\left(C'\subset A_{s_n}\,|\,C\subset A_{s_n}\right)&\le 2\sum_{k=1}^n 2^{2^k}(m_k+1)^{2^k}\sum_{p=0}^{m_k}  2^{-2^{n+1}}2^{2^{k+1}+m_{k+1}+\ldots+m_n}2^p\\
&\le 2^{1-2^{n+1}}\sum_{k=1}^n (m_k+1)^{2^k} 2^{2^{k+2}+m_{k+1}+\ldots+m_n}\sum_{p=0}^{m_k}2^p\\
&\le 2^{2-2^{n+1}}\sum_{k=1}^n (m_k+1)^{2^k} 2^{2^{k+2}+m_{k}+\ldots+m_n}\,.
\end{align*}
Using \eqref{eq:m_n}, it follows that $2^{k+2}\le m_{k-2}$ and $(m_k+1)^{2^k}\le 2^{m_{k-1}}$ for all $k$. We conclude that
\begin{align*}
&\sum_{C'\in\mathcal{C}_n}\PP\left(C'\subset A_{s_n}\,|\,C\subset A_{s_n}\right)\\
&=O(1) 2^{-2^{n+1}}\sum_{k=1}^n2^{m_{k-2}+m_{k-1}+\ldots+m_n}\\
&=O(1) 2^{-2^{n+1}+m_1+\ldots+m_n}\,.
\end{align*}
Thus, we have verified \eqref{eq:ultimate_goal}.
\end{proof}

We can now conclude the proof of Claim \ref{D}, by adapting the idea from \S\ref{subsec:B}. Recall that $Y_n$ is the number of $C\in\mathcal{C}_n$ that are contained in $A_{s_n}$. Using Proposition \ref{prop:correlation-Cantor} and \eqref{eq:expect-Cantor-Yn}, we have
\begin{align*}
	\EE(Y_n^2) &= \sum_{C,C'\in\mathcal{C}_n}  \PP(C, C'\subset A_{s_n}) \\
	&= \sum_{C\in\mathcal{C}_n} \PP(C\subset A_{s_n}) \sum_{C'\in\mathcal{C}_n} \PP(C'\subset A_{s_n}\,|\,C\subset A_{s_n})\\
	 &= O(1)2^{-2^{n+1}+s_n}\, \EE(Y_n) = O(1)E(Y_n)^2\,.
\end{align*}
The same standard argument from \S\ref{subsec:B} now implies that $Y_n\ge 1$ for all $n$ with positive probability. Condition on this event. Hence, there are lines $\ell_n\in\mathcal{L}_n$ and $C_n\in\mathcal{C}_n$, $C_n\subset \ell_n$ such that $C_n\subset A_
{s_n}$. 

Consider the Hausdorff metric on the space of compact subset of $[0,1]^2$, endowed with the dyadic metric. This space is itself compact, and so $C_n$ has a subsequence $C_{n_j}$ converging to a compact set $C\subset[0,1]^2$. Since $C_n\subset A_{s_n}$ for all $n$, it follows that $C\subset A$. Also, $C$ is contained in some line $\ell$, which is the Hausdorff limit of $\ell_{n_j}=\ell_{C_{n_j}}$.

For each $n$, if $j$ is large enough then the line containing $C_{n_j}$ is contained in the $2^{-s_n}$-dyadic interval to which $\ell$ belongs. Similarly, if $j$ is large enough then the dyadic intervals of size $2^{-s_n}$ that make up $C_{n_j}$ within $\ell_{n_j}$ also stabilize, and since $C_{n_j}\in\mathcal{C}_{n_j}$ this implies that the union of all such intervals is in $\mathcal{C}_n$.

Summarizing, with positive probability there exists a nested sequence $E_1\supset E_2\supset\ldots$ such that $\cap_{n=1}^\infty E_n\subset A$, such that $E_n$ has the form
\begin{equation}\label{eq:E_n}
	E_n=\bigcup_{j=1}^{2^n}I_j\times[k 2^{-s_n},(k+1)2^{-s_n}]\,,
\end{equation}
where $\cup_{j=1}^{2^n}I_j \times \{(k+\tfrac12)2^{-s_n}\}\in\mathcal{C}_n$.
 
Now $C=\cap_{n=1}^\infty E_n\subset A$ is a Cantor set on the horizontal line $\ell$, and it is a standard procedure to check, using \eqref{eq:m_n}, that $\mathcal{H}^{\phi}(C)=\Omega(1)$ for $\phi(r)=\frac{\log\log(1/r)}{\log(1/r)}$. Indeed, there is a natural mass distribution $\mu$ on $C$, defined by giving the same mass $2^{-n}$ to each of the dyadic cubes making up $E_n$. Let $I$ be any interval in $\ell$ of size $r\in (0,1)$. Suppose $2^{-s_{n+1}}\le r< 2^{-s_n}$. Then $I$ intersects at most $2$ of the dyadic squares making up $E_n$ and so, recalling \eqref{eq:m_n},
\[
\mu(I)= 2\cdot 2^{-n} \le O(1) \phi(2^{-s_{n+1}}) \le O(1)\phi(r).
\]
As in the proof of the mass distribution principle, we conclude that if $(I_j)$ is any cover of $C$ by intervals, then $\sum_j \phi(\diam I_j )\ge \Omega(1)$.

We have shown that, with positive probability, the fractal percolation set $A$ has positive $\mathcal{H}^{\phi}$ measure on some horizontal line. Upgrading positive probability to almost surely on non-extinction is immediate since the event
\[
    \mathcal{H}^\phi(A\cap \ell)=0\text{ for all horizontal lines }\ell
\]
is inherited.


\end{document}